\documentclass[twoside,10pt]{amsart}

\usepackage{amsthm,amsmath,amsfonts,epsfig,amssymb}
\usepackage[T1]{fontenc}
\usepackage{graphicx}
\usepackage{enumerate}
\usepackage{xy}
\usepackage{graphics}
\xyoption{all}
\entrymodifiers={+!!<0pt,\fontdimen22\textfont2>}

\newtheorem{thm}{Theorem}[section]

\newtheorem{prop} [thm]{Proposition}
\newtheorem{lem} [thm]{Lemma}
\newtheorem{cor}[thm]{Corollary}

\newtheorem{quest}[thm]{Question}
\newtheorem*{thm*}{Theorem}
\newtheorem*{quest*}{Question}
\theoremstyle{definition}
\newtheorem{defin}[thm]{Definition}
\theoremstyle{remark}
\newtheorem{rem}[thm]{Remark}

\def\R{\mathbb{R}}

\def\Z{\mathbb{Z}}
\def\N{\mathbb{N}}

\def\A{\mathbb {A}}
\def\O{\mathcal O}

\def\hom#1#2{\mathrm{Hom}(#1,#2)}
\def\homm#1#2#3{\mathrm{Hom}_{#1}(#2,#3)}

\def\exten#1#2#3#4{\mathrm{Ext}_{#2}^{#1}(#3,#4)}

\def\spec#1{\mathrm{Spec}(#1)}

\def\cha#1{\widehat {#1}}

\def\Ko#1{\widetilde K_0Sp(#1)}
\def\GW#1#2{GW^{#1}_{red}(#2)}
\def\SL#1{SL_{#1} }

\begin{document}

\title{A degree map on unimodular rows}
\author{J. Fasel}
\date{}
\email{jean.fasel@gmail.com}
\address{Jean Fasel \\
Mathematisches Institut der Universit\"at M\"unchen \\
Theresienstr. 39\\
80333 M\"unchen, Germany}

\begin{abstract}
Let $k$ be a field and let $(\A^n-0)$ be the punctured affine space. We associate to any morphism $g:(\A^n-0)\to (\A^n-0)$ an element in the Witt group $W(k)$ that we call the degree of $g$. We then use this degree map to give a negative answer to a question of M. V. Nori about unimodular rows.
\end{abstract}

\maketitle

\pagenumbering{arabic}



\section*{Introduction}

Let $R$ be a ring and let $P$ be a projective $R$-module such that $P\oplus R\simeq R^{n+1}$ for some $n\in\N$. The obvious question is to know if such an isomorphism yields an isomorphism $P\simeq R^n$. In general, this is not the case and finding general conditions for the answer to be positive is an important an extensively studied problem. Recall that one can associate to a projective module as above a unimodular row of length $n+1$, i.e. a row $(a_0,\ldots,a_n)$ such that $\sum Ra_i=R$. Let $Um_{n+1}(R)$ be the set of unimodular rows of length $n+1$. It is clear that $GL_{n+1}(R)$ acts on $Um_{n+1}(R)$ by right multiplication, and therefore so does any subgroup of $GL_{n+1}(R)$. In general, one is particularly interested by the groups $E_{n+1}(R)$ generated by elementary matrices and $SL_{n+1}(R)$, because $Um_{n+1}(R)/E_{n+1}(R)$ corresponds to some cohomotopy group in the topological situation and then carries a groups structure when $n$ is "reasonable" compared to the Krull dimension of $d$ (\cite[Theorem 4.1]{vdk}) and because $Um_{n+1}(R)/SL_{n+1}(R)$ classifies up to isomorphism projective modules $P$ such that $P\oplus R\simeq R^{n+1}$. One says that a unimodular row $v$ is completable if $v\in e_1SL_{n+1}(R)$ where $e_1=(1,0,\ldots,0)$. Thus the question to know if $P\oplus R\simeq R^{n+1}$ implies $P\simeq R^n$ reduces to the question to know if the associated unimodular row is completable. 

The first general condition on unimodular rows to be completable was given by Suslin (\cite[Theorem 2]{Suslin77c}).

\begin{thm*}
Let $R$ be an arbitrary ring, let $(a_0,\ldots,a_n)$ be a unimodular row and let $m_0,\ldots,m_n$ be positive integers. Suppose that $\prod_{i=0}^n m_i$ is divisible by $n!$. Then the row $(a_0^{m_0},\ldots,a_n^{m_n})$ is completable.
\end{thm*}  

If $R$ is an algebra over a field $k$, observe that an element $v\in Um_{n+1}(R)$ corresponds naturally to a morphism of schemes $f:\spec R\to (\A^{n+1}_k-0)$. Now the homomorphism $\phi:k[x_0,\ldots,x_n]\to k[x_0,\ldots,x_n]$ defined by $\phi(x_i)=x_i^{m_i}$ induces a morphism $\varphi:(\A^{n+1}_k-0)\to (\A^{n+1}_k-0)$ and Suslin's result reads as follows: Let $v\in Um_{n+1}(R)$ corresponding to the morphism $f:\spec R\to (\A^{n+1}_k-0)$. Then the row $w$ corresponding to $\varphi f$ is completable. In view of this, M. V. Nori asked the following very natural question (\cite{Mohan}):

\begin{quest*}[M. V. Nori]\label{nori}
Let $R=k[x_0,\ldots,x_n]$ be the polynomial ring in $n+1$ variables, where $k$ is a field. Let $\phi:R\to A$ be a $k$-algebra homomorphism such that $\sum \phi(x_i)A=A$. Let $f_0,\ldots,f_n$ such that $\mathrm{rad}(f_0,\ldots,f_n))=(x_0,\ldots,x_n)$. Assume $l(R/(f_0,\ldots,f_n))$ is divisible by $n!$. Then is the unimodular row $(\phi(f_0),\ldots,\phi(f_n))$ completable?
\end{quest*}

Mohan Kumar proved that this question has a positive answer when the base field is algebraically closed and the polynomials $f_i$ are homogeneous (\cite{Mohan}). The purpose of this article is to show that the answer to Nori's question is negative in general for $n$ odd. More precisely, we prove the following theorem (Theorem \ref{counter} in the text).

\begin{thm*}
Let $k$ be a field such that $\sqrt -1\not\in k$. Consider the unimodular row $(x_1,x_2,x_3)\in Um_3(S_3)$ and the map $g:(\A^3-0)\to (\A^3-0)$ defined by the algebra homomorphism $\varphi:k[x_1,x_2,x_3]\to k[x_1,x_2,x_3]$ given by $\varphi(x_1)=x_1^2-x_2^2$, $\varphi(x_2)=x_1x_2$ and $\varphi(x_3)=x_3$. Then $k[x_1,x_2,x_3]/(x_1^2-x_2^2,x_1x_2,x_3)$ is of length $4$, but $(x_1^2-x_2^2,x_1x_2,x_3)\in Um_3(S_3)$ is not completable.
\end{thm*}

Our guess is that Nori's question has an affirmative answer when $n$ is even. We now describe our strategy to prove this result. Consider the affine $k$-algebra $A_n=k[x_1,\ldots,x_n,y_1,\ldots,y_n]/(\sum x_iy_i-1)$ and set $S_n=\spec {A_n}$. There is a natural projection $p:S_n\to (\A^n_k-0)$ with affine fibres. Now it is not hard to see that any unimodular row $f:\spec R\to (\A^n_k-0)$ factors through $S_n$, i.e. there exists a morphism $f^\prime:\spec R\to S_n$ such that $f=pf^\prime$. This is the reason why we call $S_n$ the unimodular affine scheme. To prove a general property of the unimodular rows of length $n$ over any $k$-algebra $R$, it suffices then to prove it over $S_n$.

Suppose for a while that $k=\R$. Let $f:S_n\to (\A^n_{\R}-0)$ be any morphism. Considering only the real closed points, and using the fact that the projection $p:S_n\to (\A^n_k-0)$ has affine fibers, we are left up to homotopy with a polynomial map $f:(\R^n-0)\to (\R^n-0)$. We can consider its (Brouwer) degree, defined as the sum of the signs of the Jacobian of $f$ at all preimages (under $f$) near $0$ of a regular value of $f$ near $0$. Since the Witt group of $\R$ is equal to $\Z$, we can see the degree as an element of $W(\R)$. It is worth noting that the degree of two homotopic maps are the same, and that the degree of the constant map is $0$. Thus a map with a non constant degree cannot be homotopic to a constant map. Translating this rather easy topological situation to an algebraic situation will be our motivation for the rest of the paper. 

We begin with the definition of a symbol $\phi:Um_n(R)\to GW^{n-1}_{red}(R)$, where the latter is a quotient of  the derived Grothendieck-Witt groups of $R$, as defined by Walter (\cite{Wa}, see also \cite{Baba}). This definition requires some knowledge of Grothendieck-Witt groups, and we therefore spend a few sections to remind some basic facts about these groups. The group $GW^{n-1}_{red}(R)$ being homotopy invariant, this symbol factors through the action of the group $E_n(R)$ generated by elementary matrices to give a symbol $\phi:Um_n(R)/E_n(R)\to GW^{n-1}_{red}(R)$ for $n\geq 3$. 

If $n$ is odd, the symbol yields a symbol $\Phi:Um_n(R)/SL_n(R)\to GW^{n-1}_{red}(R)$ while this is not the case for $n$ even. This rather odd fact will probably be explained when the computation of the Grothendieck-Witt groups of $SL_n$ will be fully achieved (we expect a different answer depending on the parity of $n$). We then compute the Grothendieck-Witt groups of $S_n$ to get $GW^{n-1}_{red}(S_n)=W(k)$, the Witt group of the base field. Under this identification $\Phi(p)=\langle 1\rangle$, where $p:S_n\to (\A^{n}_k-0)$ is the projection. 

In Section \ref{degree_map}, we associate to any morphism $f:(\A^n_k-0)\to (\A^n_k-0)$ a degree generalizing the (Brouwer) degree. This degree map take coefficient in the Witt group $W(k)$ of the base field $k$, and we ingeniously denote by $\mathrm{deg}(f)$ the degree of the map $f$. The basic fact here is that $\Phi(f\circ p)=\mathrm{deg}(f)$ in $W(k)$! This is exactly the obstruction we need to exhibit our counter-example to Nori's question.

In view of the example provided by Theorem \ref{counter}, we propose a strengthened conjecture (Question \ref{nori+} in the text) which is equivalent to the original conjecture when the base field $k$ is algebraically closed. 

\begin{quest*}
Let $R$ be a $k$-algebra, $n\in\N$ be odd and let $f:\spec R\to (\A^n-0)$ be a unimodular row. Let $\varphi:k[x_1,\ldots,x_n]\to k[x_1,\ldots,x_n]$ such that $\varphi(\mathfrak m)\subset \mathfrak m$ and that $(n-1)!$ divides the length of $k[x_1,\ldots x_n]/\varphi(\mathfrak m)$. Let $g:(\A^n-0)\to (\A^n-0)$ be the morphism induced by $\varphi$. If the degree $\mathrm{deg}(g)=0$, then the unimodular row $gf:\spec R\to (\A^n-0)$ is completable.
\end{quest*}

\subsection*{Conventions}
Any scheme is of finite type and separated over a field of characteristic different from $2$.

\subsection*{Acknowledgments}
I'm grateful to Stefan Gille for communicating me the correct signs in the definition of the symmetric form on a Koszul complex. It is a pleasure to thank Ravi Rao for bringing the question to my notice and Sarang Sane for useful conversations. I also wish to thank Wilberd van der Kallen for his remarks on a preliminary version of the paper. This work was supported by the Swiss National Science Foundation, grant PAOOP2\_129089. 


\section{Unimodular rows}\label{um}
Let $R$ be a $k$-algebra. A unimodular row of length $n\geq 2$ is a surjective homomorphism $R^n\to R$. We denote by $Um_n(R)$ the set of unimodular rows of length $n$. The group $GL_n(R)$ acts on $Um_n(R)$ by left composition, and so do any subgroup. We say that $v\in Um_n(R)$ is completable if it is the first row of a matrix in $SL_n(R)$. 

Observe that $Um_n(R)=\hom {\spec R}{\A^n-0}$. Let $r:\SL n\to (\A^n-0)$ be the projection on the first row. In this setting, a unimodular row $f:\spec R\to (\A^n-0)$ is completable if $f$ factors through $\SL n$, i.e if there exists a morphism $f^\prime:\spec R\to \SL n$ such that the diagram
$$\xymatrix{\spec R\ar[r]^-{f^\prime}\ar[rd]_-f & \SL n\ar[d]^-r \\
 & (\A^n-0)}$$
commutes.
\subsection{Steinberg symbols}

Let $R$ be a ring. A Steinberg symbol is a pair $(\rho,G)$, where $G$ is a group and $\rho:Um_n(R)\to G$ is a map such that the following relations hold:
\begin{enumerate}[(1)]
\item $\rho(v)=\rho(vE)$ for any $E\in E_n(R)$.
\item $\rho(x,v_2,\ldots,v_n)\rho(1-x,v_2,\ldots,v_n)=\rho(x(1-x),v_2,\ldots,v_n)$ if $(x,v_2,\ldots,v_n)$ and $(1-x,v_2,\ldots,v_n)$ are unimodular. 
\end{enumerate} 

It is clear that a universal Steinberg symbol exists. We denote it by $(St_n(R),st)$. Observe that by definition a Steinberg symbol $\rho:Um_n(R)\to G$ induces a map $\rho:Um_n(R)/E_n(R)\to G$. If $R$ is a ring of Krull dimension $d$ with $2\leq d\leq 2n-4$, then W. van der Kallen proved that the universal Steinberg symbol $(St_n(R),st)$ induces a bijection $st:Um_n(R)/E_n(R)\to St_n(R)$ (\cite[Theorem 3.3]{vdk2}). Moreover, $St_n(R)$ is abelian in this situation (\cite[Remark 3.4]{vdk2}), and therefore $st$ endows $Um_n(R)/E_n(R)$ with the structure of an abelian group.

\subsection{Unimodular rows and projective modules}

For any unimodular element $a=(a_1,\ldots,a_n)\in Um_n(R)$ we denote by $P(a)$ the cokernel of the homomorphism $a^t:R\to R^n$. We have then an exact sequence
$$\xymatrix{0\ar[r] & R\ar[r]^-{a^t} & R^n\ar[r] & P(a)\ar[r] & 0.}$$
Usually, one defines $P(a)$ to be the kernel of $a:R^n\to R$ and it is clear that our definition is the dual of this one.
If $n$ is odd, then $s(a):R^n\to R$ given by the row $(-a_2,a_1,\ldots,-a_{n-1},a_{n-2},0)$ has the property that $s(a)(a^t)=0$. It induces a homomorphism $P(a)\to R$ that we still denote $s(a)$. If $n$ is even, then we can define $s(a)$ by the row $(-a_2,a_1,\ldots,-a_{n},a_{n-1})$ giving a surjective homomorphism $s(a):P(a)\to R$.

The sequence
$$\xymatrix{0\ar[r] & R\ar[r]^-{a^t} & R^n\ar[r] & P(a)\ar[r] & 0}$$
induces an isomorphism $\chi(a):\det P(a)\simeq R\cdot e_1\wedge\ldots\wedge e_n$, where $e_1,\ldots,e_n$ is the usual basis of $R^n$. If $b=(b_1,\ldots,b_n)$ is such that $ba^t=1$ and $f_i$ is the image of $e_i$ in $P(a)$ for $1\leq i\leq n$, a straightforward computation shows that $\chi(a)^{-1}(e_1\wedge\ldots\wedge e_n)=\sum_{i=1}^n(-1)^ib_i\cdot f_1\wedge\ldots\wedge f_{i-1}\wedge f_{i+1}\wedge\ldots\wedge f_n$. We denote by $\omega(a)$ this generator of $\det P(a)$. Since $a_i\omega(a)=(-1)^if_1\wedge\ldots\wedge f_{i-1}\wedge f_{i+1}\wedge\ldots\wedge f_n$ for all $1\leq i\leq n$, we see that $\omega(a)$ is independent of the choice of $b$. Indeed, if $c$ is such that $ca^t=1$, then 
$$\omega(a)=(\sum_{i=1}^n c_ia_i)\omega(a)=\sum_{i=1}^n(-1)^ic_i f_1\wedge\ldots\wedge f_{i-1}\wedge f_{i+1}\wedge\ldots\wedge f_n.$$
%


\section{Grothendieck-Witt groups}\label{gw}

\subsection{The basics}
Let $X$ be a scheme over $\spec k$. For any $j\in\Z$ and any line bundle $L$ over $X$, we denote by $GW^j(X,L)$ the $j$-th Grothendieck-Witt group of the derived category of bounded complexes of coherent locally free $\O_X$-modules, with duality induced by the functor $\homm {\O_X}{\_}L$ (\cite[\S 2]{Wa}).  We denote by $W^j(X,L)$ the Witt groups of this category (\cite[Definition 1.4.3]{Baba}). By definition, there is an exact sequence (\cite[Theorem 2.6]{Wa})
$$\xymatrix{K_0(X)\ar[r] & GW^j(X,L)\ar[r] & W^j(X,L)\ar[r] & 0.}$$
The Grothendieck-Witt groups are contravariant: If $f:Y\to X$ is a morphism of schemes, then there are homomorphisms $f^*:GW^j(X,L)\to GW^j(Y,f^*L)$. In particular, the structural morphism $p:X\to \spec k$ induces a homomorphism $p^*:GW^j(k)\to GW^j(X)$ for any $j\in\Z$. We denote by $GW^j_{red}(X)$ its cokernel (which is a direct factor of $GW^j(X)$ if $X$ has a rational point).

\subsection{Transfers}
Suppose that $i:Y\to X$ is a closed immersion. Then we can consider the category of bounded complexes of coherent locally free $\O_X$-modules whose homology is supported on $Y$. We will denote its Grothendieck-Witt groups by $GW^j_Y(X,L)$. We can sometimes identify the groups with support on $Y$ with the groups associated to $Y$. This procedure is usually called d\'evissage. In the special situation where $X$ and $Y$ are smooth, the d\'evissage can be described as follows: Let $\overline i:(Y,\O_Y)\to (X,i_*\O_Y)$ be the morphism of ringed spaces induced by $i$. Suppose that $i$ is of pure codimension $d$. If $L$ is a line bundle over $X$, then $N:=\overline i^*\exten d{\O_X}{i_*\O_Y}{i^*L}$ is a line bundle over $Y$ and there is a transfer morphism (\cite[Theorem 6.2]{CH})
$$i_*:GW^j(Y,N)\to GW^{j+d}_Y(X,L).$$
At the level of the Witt groups, this homomorphism is in fact an isomorphism (\cite[Theorem 3.2]{Gi3}). This fact is also true for Grothendieck-Witt groups, but we don't use it here. In case $N$ is trivial, $i_*$ becomes an isomorphism
$$i_*:GW^j(Y)\to GW^{j+d}_Y(X,L).$$
Observe that $i_*$ now depends on a trivialization isomorphism $\O_Y\simeq N$.

\subsection{Finite length modules and d\'evissage}\label{finite_length}
Let $R$ be a regular local ring of dimension $d$. Let $\mathcal M_{fl}(R)$ be the (abelian) category of finite length $R$-modules. The functor $M\mapsto \cha M:=\exten dRMR$ endows $\mathcal M_{fl}(R)$ with a duality with canonical isomorphism defined as follows: Let $\pi:\mathcal P\to M$ be a projective resolution of $M$. The canonical identification of $\mathcal P$ with its bidual $\mathcal P^{\vee\vee}$ yields an isomorphism $\eta:M\to \cha{\cha M}$ which is independent of the projective resolution (\cite[Theorem 3.3.2]{Fa1} for instance). We set $\varpi:=(-1)^{d(d-1)/2}\eta$. Having a duality functor and a canonical isomorphism, we can define the Witt group of finite length modules $W^{fl}(R)$ following \cite{QSS}. 

Suppose that $X$ is smooth of dimension $d$ and that $x\in X$ is a closed point. We consider the Witt group $W^d_{x}(X)$. By the above section, we know that there is an isomorphism
$$i_*:W(k(x),\exten d{\O_X}{k(x)}{\O_X})\to W^d_{x}(X).$$ 
We can also interpret this isomorphism using finite length $\O_{X,x}$-modules following \cite[\S 6]{BW}. 

Observe first that localizing at $x$ induces an isomorphism $W^j_x(X)\to W^j_x(\O_{X,x})$ for any $j\in\N$. Now $W^{fl}(\O_{X,x})\simeq W^d_x(\O_{X,x})$ by \cite[Theorem 6.1, Proposition 6.2]{BW}. The map is defined by sending a finite length module (endowed with a form $\phi$) to a projective resolution of this module (endowed with a form induced by $\phi$). Finally, there is a canonical isomorphism $W(k(x),\exten d{\O_X}{k(x)}{\O_X})\to W^{fl}(\O_{X,x})$ (\cite[Appendix E.2]{Fa1}). The isomorphism $i_*$ is the composition
$$\xymatrix@C=1.2em{W(k(x),\exten d{\O_X}{k(x)}{\O_X})\ar[r] &  W^{fl}(\O_{X,x})\ar[r] & W^d_x(\O_{X,x})\ar[r] & W^d_x(X).}$$

\subsection{Euler classes}

Let $X$ be a scheme and $\mathcal E$ be a rank $n$ coherent locally free $\O_X$-module over $X$. Let $s:\mathcal E\to \O_X$ be any section (possibly trivial). Then we can consider the Koszul complex $Kos(s)$ associated to $s$. For any $1\leq i\leq n$, we have isomorphisms $\varphi_i:\wedge^i\mathcal E\to \homm {\O_X}{\wedge^{n-i}\mathcal E}{\det \mathcal E}$ defined by $\varphi_i(p)(q)=p\wedge q$. Let 
$$\rho(s):Kos(s)\to T^n\homm {\O_X}{Kos(s)}{\det \mathcal E}$$
be the isomorphism defined in degree $i$ by $\rho(s)_i=(-1)^{in+i(i-1)/2+n(n-1)/2}\varphi_i$ (note that the $\rho_i$ given in \cite[Remark 4.2]{BG} do not define a morphism of complexes). It turns out that $\rho(s)$ is symmetric for the $n$-th shifted duality, and therefore it defines an element in $GW^n(X,\det \mathcal E)$.

\begin{lem}
Suppose that $X$ is regular. let $s:\mathcal E\to \O_X$ and $s^\prime:\mathcal E\to \O_X$ be two sections. Then $\rho(s)=\rho(s^\prime)$ in $GW^n(X,\det \mathcal E)$. 
\end{lem}

\begin{proof}
Let $p:X\times \A^1\to X$ be the projection. We consider the section 
$$(tp^*s+(1-t)p^*s^\prime):p^*\mathcal E\to \O_{X\times\A^1}.$$
Since $X$ is regular, $p^*:GW^n(X,\det \mathcal E)\to GW^n(X\times \A^1,p^*\det \mathcal E)$ is an isomorphism (\cite[Proposition 1.1]{Fa4}). There exists then $\alpha\in GW^n(X,\det \mathcal E)$ such that $p^*\alpha=\rho(tp^*s+(1-t)p^*s^\prime)$. Evaluating at $t=0$ and $t=1$, we get $\alpha=\rho(s)=\rho(s^\prime)$.
\end{proof}

\begin{defin}
We call Euler class of $\mathcal E$ the element $\rho(s)$ in $GW^n(X,\det \mathcal E)$ for any section $s:\mathcal E\to \O_X$. We denote it by $e(\mathcal E)$.
\end{defin}

If $f:Y\to X$ is a morphism of regular schemes, then it is clear from the definition that $e(f^*\mathcal E)=f^*e(\mathcal E)$.

Let $g:\mathcal E\to \mathcal E$ be an isomorphism such that $\det f=1$ and let $p:E\to X$ be the total space of $\mathcal E$. If we also denote by $g:E\to E$ the morphism induced by $g$, then we see from the definition of the Euler class that $g^*e(p^*\mathcal E)=e(p^*\mathcal E)$. It follows that the Euler class is invariant under the action of $SL(\mathcal E)$.

Suppose next that $\chi:\det \mathcal E\to \O_X$ is an isomorphism. It induces an isomorphism $GW^n(X,\det \mathcal E)\to GW^n(X)$. We denote by $e(\mathcal E,\chi)$ the image of $e(\mathcal E)$ under this isomorphism. It depends of course on the isomorphism $\chi$, but it is still invariant under the action of $SL(\mathcal E)$.

\subsection{Some computations}

Let $X$ be a scheme. If $x\in\O_X(X)$ is a global section, we denote by $Z(x)$ the vanishing locus of $X$, which is a closed subset in $X$ (since $X$ is noetherian by our conventions). 

Let $(a_1,\ldots,a_n)\in \O_X(X)$ and let $Kos(a_1,\ldots,a_n)$ be the Koszul complex associated to the (non necessarily regular) sequence of global sections $(a_1,\ldots,a_n)$. As in the previous section, this complex is endowed with a symmetric isomorphism
$$\rho(a_1,\ldots,a_n):Kos(a_1,\ldots,a_n)\to T^{n}Kos(a_1,\ldots,a_n)^\vee.$$
Let $Z:=\cap_{i=1}^n Z(a_i)$ be the closed subset of $X$ where the sections all vanish. Since the homology of $Kos(a_1,\ldots,a_n)$ is concentrated on $Z$, the above symmetric isomorphism defines an element in $GW^n_Z(X)$ (and $W^n_Z(X)$) that we still denote by $\rho(a_1,\ldots,a_n)$. 

We can also consider the form $\rho(a_1,\ldots,a_{n-1})$ on $Kos(a_1,\ldots,a_{n-1})$ whose homology is supported on $Z^\prime:=\cap_{i=1}^{n-1}Z(a_i)$. If $U=X-Z$, then 
$$(Z(a_n)\cap U)\cap (Z^\prime\cap U)=\emptyset.$$
It follows that  
$$a_n\rho(a_1,\ldots,a_{n-1}):Kos(a_1,\ldots,a_n)\to T^{n}Kos(a_1,\ldots,a_n)^\vee$$
is an isomorphism on $U$, and we denote its class in $GW^{n-1}(U)$ by $\theta(a_1,\ldots,a_n)$. The following result is a straightforward consequence of \cite[Remark 4.2]{BG} and the Leibnitz formula for symmetric isomorphisms (\cite[Theorem 5.2]{Ba2}):

\begin{lem}\label{connecting}
Let $\partial:GW^{n-1}(U)\to W^n_{Z}(X)$ be the connecting homomorphism in the localization sequence associated to $U\subset X$. Then $\partial(\theta(a_1,\ldots,a_n))=\rho(a_1,\ldots,a_n)$.
\end{lem}

\begin{lem}\label{coefficient}
For any $n\in\N$, we have $GW^{n-1}_{red}(\A^n-0)=W(k)\cdot \theta(x_1,\ldots,x_n)$.
\end{lem}

\begin{proof}
The long exact sequence of localization reads as
$$\xymatrix@C=0.9em{\ldots\ar[r] & GW_{\{0\}}^{n-1}(\A^n)\ar[r] & GW^{n-1}(\A^n)\ar[r] & GW^{n-1}(\A^n-0)\ar[r]^-\partial & W^n_{\{0\}}(\A^n)\ar[r] & \ldots       }$$
By homotopy invariance of Grothendieck-Witt groups (\cite[Proposition 1.1]{Fa4}), we see that $GW^{n-1}(\A^n)=GW^{n-1}(k)$. Using this, we see that the choice of a rational point in $\A^n-0$ gives a splitting of the map $GW^{n-1}(\A^n)\to GW^{n-1}(\A^n-0)$. We then get a split exact sequence
$$\xymatrix{0\ar[r] & GW^{n-1}(k)\ar[r] & GW^{n-1}(\A^n-0)\ar[r]^-\partial & W^n_{\{0\}}(\A^n)\ar[r] & 0}$$
showing that $GW^{n-1}_{red}(\A^n-0)\simeq W^n_{\{0\}}(\A^n)$.

By d\'evissage, $W^n_{\{0\}}(\A^n)=W(k)\cdot \rho(x_1,\ldots,x_n)$. We can use Lemma \ref{connecting} to conclude.
\end{proof}

\begin{rem}\label{orientation}
Let $R=k[x_1,\ldots,x_n]$ and $\mathfrak m=(x_1,\ldots,x_n)$. Observe that the above d\'evissage isomorphism $W^n_{\{0\}}(\A^n)\simeq W(k)$ amounts to choose a generator of $\exten nR{R/\mathfrak m}R$. When we write $W^n_{\{0\}}(\A^n)=W(k)\cdot \rho(x_1,\ldots,x_n)$, we implicitly choose the Koszul complex associated to the sequence $(x_1,\ldots,x_n)$ as a generator.
\end{rem}

\begin{rem}\label{action}
The action of $W(k)$ on $\theta(x_1,\ldots,x_n)$ reads as follows: If $\alpha\in k^\times$, then $\langle \alpha\rangle\cdot \theta(x_1,\ldots,x_n)=\theta(x_1,\ldots,x_{n-1},\alpha x_n)$. 
\end{rem}

We will need the following lemma in the next section:

\begin{lem}\label{transfer}
For any $n\geq 1$, Let $j:\A^{n-1}\to \A^n$ be the inclusion obtained by setting $x_n=0$. Let $j^\prime:(\A^{n-1}-0)\to (\A^n-0)$ be the induced closed immersion. Then the transfer homomorphism
$(j^\prime)_*:GW^{n-2}(\A^{n-1}-0)\to GW^{n-1}(\A^n-0)$
induces an isomorphism
$$(j^\prime)_*:GW^{n-2}_{red}(\A^{n-1}-0)\to GW^{n-1}_{red}(\A^n-0)$$
\end{lem}

\begin{proof}
The transfer homomorphism associated to $j$ induces a commutative diagram
$$\xymatrix@C=1.5em{0\ar[r] & GW^{n-2}(\A^{n-1})\ar[r]\ar[d]^-{j_*} & GW^{n-2}(\A^{n-1}-0)\ar[r]^-\partial\ar[d]^-{j^\prime_*} & W^{n-1}_{\{0\}}(\A^{n-1})\ar[r]\ar[d]^-{j_*} & 0\\
0\ar[r] & GW^{n-1}(\A^{n})\ar[r] & GW^{n-1}(\A^{n}-0)\ar[r]^-\partial & W^{n}_{\{0\}}(\A^{n})\ar[r] & 0}$$
where the lines are the exact sequence of localization obtained in the proof of the above lemma. By d\'evissage, the homomorphism $j_*$ on the right is an isomorphism. 
\end{proof}


\section{The symbol $\Phi$}\label{symbolic}

Let $R$ be a $k$-algebra. Any unimodular element $(a_1,\ldots,a_n)$ yields a morphism $f:\spec R\to (\A^n-0)$. Using the above section, we get a map
$$\phi:Um_n(R)\to \GW {n-1}R$$
defined by $\phi(f)=f^*(\theta(x_1,\ldots,x_n))$.

\begin{lem}\label{wms}
Suppose that $R$ is regular. Then the map $\phi:Um_n(R)\to \GW {n-1}R$ is a Steinberg symbol.
\end{lem}

\begin{proof}
Let $v$ and $w$ be unimodular rows. We say that $v$ and $w$ are elementarily homotopic if there is a morphism $F:\spec R\times \A^1\to (\A^n-0)$ such that $F(0)=v$ and $F(1)=w$. Since $\GW {n-1}R$ is homotopy invariant, it is clear that $\phi(v)=\phi(w)$ if $v$ and $w$ are elementarily homotopic. We conclude from \cite[Theorem 2.1]{Fa3} that $\phi$ induces a map 
$$\phi:Um_n(R)/E_n(R)\to \GW {n-1}R.$$
We next prove that relation (2) holds. 

Let $a=(a_1,\ldots,a_{n-1})\in R^{n-1}$ be such that $(a,x)$ and $(a,1-x)$ are unimodular. Observe that the endomorphism $\alpha$ of $Kos(a)\oplus Kos(a)$ given by the matrix
$$\begin{pmatrix} 1 & x-1\\ 1 & x\end{pmatrix}$$
is an automorphism. A straightforward computation shows that 
$$(T^{n-1}\alpha^\vee)(x\rho(a)\bot (1-x)\rho(a))\alpha=\rho(a)\bot x(1-x)\rho(a)$$
where $\bot$ denotes the orthogonal sum. It follows that 
$$x\rho(a)+(1-x)\rho(a)=\rho(a)+x(1-x)\rho(a)$$
in $\GW {n-1}R$. But $\rho(a)=\phi(a,1)=0$ (observe that $\rho(a)=0$ in $\GW {n-1}R$, but that it is in general not trivial in $GW^{n-1}_{V(a)}(R)$). 
\end{proof}

In particular, $\phi$ induces a map $\phi:Um_n(R)/E_n(R)\to \GW {n-1}R$. In case $n$ is odd, we next show that $\phi$ factors through the action of $SL_n(R)$.



\begin{prop}\label{euler}
Let $R$ be a regular $k$-algebra and $n$ be an odd integer. Then $\phi(a)=e(P(a),\chi(a))$ for any $a\in Um_n(R)$.
\end{prop}

\begin{proof}
Let $a\in Um_n(R)$ and $P(a)$ its associated stably free module. Consider the section $s(a):P(a)\to R$ defined in Section \ref{um}. Let $Kos(s(a))$ be the Koszul complex associated to this section and 
$$\rho(a):Kos(s(a))\to T^{n-1}Kos(s(a))^\vee$$
the symmetric isomorphism defining the Euler class.

Let $u:R^{n-1}\to P(a)$ defined by $u(e_{2j})=f_{2j-1}$, $u(e_{2j-1})=-f_{2j}$ for any $1\leq j\leq (n-1)/2$. Taking exterior powers, we get homomorphisms 
$$\wedge^iu:\wedge^i(R^{n-1})\to \wedge^i(P(a))$$
and it is not hard to check that this induces a morphism of complexes 
$$U:Kos(a_1,\ldots,a_{n-1})\to Kos(s(a)).$$
The homology of both these complexes are concentrated on the closed subset $V(a_1,\ldots,a_{n-1})$ of $\spec R$. Therefore $U$ is a quasi-isomorphism on the open complement of $V(a_1,\ldots,a_{n-1})$. Since $a$ is unimodular, $V(a_1,\ldots,a_{n-1})$ is included in the principal open subset $D(a_n)$ of $\spec R$. But $u$ is an isomorphism on $D(a_n)$ and it follows that $U$ is a quasi-isomorphism. 

We now prove that $U$ is in fact an isometry between $\rho(a)$ and $\theta(a_1,\ldots,a_n)$. To do this, consider the following diagram
$$\xymatrix@C=4em{\wedge^i R^{n-1}\ar[r]^-{\wedge^iu}\ar[d] & \wedge^iP(a)\ar[d]^-{\varphi_i} \\
(\wedge^{n-1-i}R^{n-1})^\vee & (\wedge^{n-1-i}P(a))^\vee\ar[l]^-{(\wedge^{n-1-i}u)^\vee}}$$
Since $-f_{2j}\wedge f_{2j-1}=f_{2j-1}\wedge f_{2j}$, we get
$$(\wedge^{n-1-i}u)^\vee\varphi_i\wedge^iu)(e_{m_1}\wedge\ldots \wedge e_{m_i})(e_{r_1}\wedge\ldots\wedge e_{r_{n-1-i}})=(-1)^{sign(\sigma)}\chi(a)(f_1\wedge\ldots\wedge f_{n-1})$$
if $\sigma$ is a permutation such that $\sigma(m_1,\ldots,m_i,r_1,\ldots,r_{n-1-i})=(1,\ldots,n-1)$ and
$$(\wedge^{n-1-i}u)^\vee\varphi_i\wedge^iu)(e_{m_1}\wedge\ldots \wedge e_{m_i})(e_{r_1}\wedge\ldots\wedge e_{r_{n-1-i}})=0$$
if $e_{m_1}\wedge\ldots \wedge e_{m_i}\wedge e_{r_1}\wedge\ldots\wedge e_{r_{n-1-i}}=0$. It follows that $U$ induces an isometry between $\rho(a)$ and $\theta(a_1,\ldots,a_n)$. Whence the result.
\end{proof}

\begin{cor}\label{symbol}
Let $R$ be a regular $k$-algebra and $n$ be an odd integer. The map $\phi:Um_n(R)\to \GW {n-1}R$ induces a map
$$\Phi:Um_n(R)/SL_n(R)\to \GW {n-1}R.$$
\end{cor}

\begin{proof}
The Euler class is invariant under the action of $SL(P(a))$.
\end{proof}

\begin{rem}
If $R$ is regular of Krull dimension $2$ then $Um_3(R)/SL_3(R)$ has an abelian group structure such that $\Phi$ is an injective homomorphism, with cokernel the Chow group $CH^2(X)$ (\cite[Theorem 7.3]{BS2}).
\end{rem}

\begin{rem}
In case $n$ is even, then $\phi$ doesn't factor through the action of $SL_n$ as easily seen by considering $n=2$. 
\end{rem}

Consider the affine scheme $S_n=\spec {k[x_1,\ldots,x_n,y_1,\ldots,y_n]/(\sum x_iy_i-1)}$. If $R$ is any $k$-algebra, $a=(a_1,\ldots,a_n)$ is a unimodular row and $b=(b_1,\ldots,b_n)$ is such that $ba^t=1$, then we obtain a morphism $\spec R\to S_n$ mapping $x_i$ to $a_i$ and $y_i$ to $b_i$. Conversely, any morphism $\spec R\to S_n$ yields such an $a$ and $b$. We therefore call $S_n$ the \emph{unimodular affine scheme}.

The ring homomorphism $k[x_1,\ldots,x_n]\to k[x_1,\ldots,x_n,y_1,\ldots,y_n]/(\sum x_iy_i-1)$ given by $x_i\mapsto x_i,y\mapsto y_1$ for any $1\leq i\leq n$ induces a morphism $p_n:S_n\to (\A^n-0)$, which is easily checked to have affine fibres. It follows then from the Mayer-Vietoris sequence (\cite[Theorem 1]{Sch}) that $p_n^*:GW^{n-1}(\A^n-0)\to GW^{n-1}(S_n)$ is an isomorphism, showing that from a cohomological viewpoint the choice of a section of a unimodular row is not important. 

\begin{prop}
If $n\geq 3$ is odd, the map $p_n:S_n\to (\A^n-0)$ induces a surjective map
$$\Phi:Um_n(S_n)/SL_n(S_n)\to W(k)$$
with $\Phi(x_1,\ldots,x_n)=\langle 1\rangle$. 
\end{prop}

\begin{proof}
In view of Lemma \ref{coefficient} and Proposition \ref{symbol}, it remains only to show that $\Phi$ is surjective. It is obviously sufficient to prove that $\phi:Um_n(S_n)/E_n(S_n)\to W(k)$ is surjective to conclude.

Suppose that $n>3$. Setting $x_i=0$ and $y_i=0$ for $i>3$, we get closed immersions $j^\prime:(\A^3-0)\to (\A^n-0)$ and $u:S_3\to S_n$ such that the diagram
$$\xymatrix{S_3\ar[r]^-u\ar[d]_-{p_3} & S_n\ar[d]^-{p_n}\\
(\A^3-0)\ar[r]_-{j^\prime} & (\A^n-0)}$$
commutes. Now we can define a map
$$\beta:Um_3(S_3)/E_3(S_3)\to Um_n(S_n)/E_n(S_n)$$
by $\beta(\overline a_1,\overline a_2,\overline a_3)=(a_1,a_2,a_3,x_4,\ldots,x_n)$, where $a_i$ are any lifts of $\overline a_i$. Seeing $p_3$ and $p_n$ as unimodular rows, it is easily checked that $\beta(p_3)=p_n$. It follows that the following diagram
$$\xymatrix{Um_3(S_3)/E_3(S_3)\ar[r]^-\beta\ar[d]_-\phi & Um_n(S_n)/E_n(S_n)\ar[d]^-\phi \\
GW^2_{red}(S_3)\ar[r]_-{u_*} & GW^{n-1}_{red}(S_n)}$$
commutes. Since $u_*$ is an isomorphism by Lemma \ref{transfer}, it suffices to check that $\phi:Um_3(S_3)/E_3(S_3)\to W(k)$ is surjective to conclude. 

Recall that $\Ko {S_3}$ is defined by the exact sequence 
$$\xymatrix{0\ar[r] & K_0(k)\ar[r]^-{p^*H} & GW^2(S_3)\ar[r] & \Ko {S_3}\ar[r] & 0}$$
where $H:K_0(k)\to GW^2(k)$ is the hyperbolic functor (see \cite[\S 2.1]{Fasel10} for instance) and $p^*:GW^2(k)\to GW^2(S_3)$ is the pull-back homomorphism. Since $H$ is an isomorphism, it follows that there is a natural isomorphism $\Ko {S_3}\to GW^2_{red}(S_3)$. Under this identification, the symbol $\phi$ coincides with the Vaserstein symbol
$$V:Um_3(S_3)/E_3(S_3)\to \Ko {S_3}$$
defined in \cite[Theorem 5.2]{VS} (see also \cite[\S 5]{Bass75}). In particular, this shows that $\phi$ satisfies the Vaserstein rule: For any unimodular rows $(a_1,a_2,a_3)$ and $(a_1,b_2,b_3)$, we have
$$\phi(a_1,a_2,a_3)+\phi(a_1,b_2,b_3)=\phi(a_1, \begin{pmatrix} a_2 & a_3\end{pmatrix}\cdot \begin{pmatrix} b_2 & b_3\\ -c_3 & c_2\end{pmatrix})$$
where $c_2,c_3$ are such that $b_2c_2+b_3c_3\equiv 1 \pmod {a_1}$. This proves that for any $\alpha_1,\ldots,\alpha_n\in k^\times$, there exists $b,c\in k[x_1,x_2,x_3,y_1,y_2,y_3]/(\sum x_iy_i-1)$ such that 
$$\phi(x_1,x_2,\alpha_1x_3)+\ldots+\phi(x_1,x_2,\alpha_nx_3)=\phi(x_1,b,c).$$
Remark \ref{action} then shows that $\langle \alpha_1,\ldots,\alpha_n\rangle$ is in the image of $\phi$.

\end{proof}

\begin{rem}
If $n\geq 5$ is an odd integer, the map $\Phi$ is not injective. Indeed, the proof of the above proposition shows that $\Phi$ satisfies the Vaserstein rule. It follows from \cite[lemma 3.5 (v)]{vdk} that $\Phi(a_1,a_2,\ldots,a_n)+\Phi(b_1^2,a_2,\ldots,a_n)=\Phi(a_1b_1^2,a_2,\ldots,a_n)$ for any unimodular rows $(a_1,a_2,\ldots,a_n)$ and $(b_1,a_2,\ldots,a_n)$. Using this and \cite[Lemma 3.5 (iii)]{vdk}, we see that 
$$\Phi(x_1^2,x_2,\ldots,x_n)=\Phi(x_1,x_2,\ldots,x_n)+\Phi(-x_1,x_2,\ldots,x_n)$$
By Remark \ref{action}, this reads as
$$\Phi(x_1^2,x_2,\ldots,x_n)=\langle 1,-1\rangle \Phi(x_1,x_2,\ldots,x_n)=0$$
But it is well-known that $(x_1^2,x_2,\ldots,x_n)$ is not completable (\cite{Suslin91}). Hence $\Phi$ is not injective.
\end{rem}

\begin{quest}
Is the map $\Phi:Um_3(S_3)/SL_3(S_3)\to W(k)$ bijective?
\end{quest}


\section{The degree map}\label{degree_map}

\subsection{The map}

Let $n\geq 2$ and $g:(\A^n-0)\to (\A^n-0)$ be a morphism. Taking global sections, we see that there is a commutative diagram
$$\xymatrix{(\A^n-0)\ar[r]^-g\ar[d]_-i & (\A^n-0)\ar[d]^-i \\
\A^n\ar[r]_-{g^\prime} & \A^n}$$
where $i:(\A^n-0)\to \A^n$ is the inclusion, and $g^\prime:\A^n\to \A^n$ is the morphism associated to some $\varphi:k[x_1,\ldots,x_n]\to k[x_1,\ldots,x_n]$ with $\varphi(\mathfrak m)\subset \mathfrak m$. We then get a commutative diagram
$$\xymatrix{0\ar[r] & GW^{n-1}(\A^n)\ar[r]^-{i^*}\ar[d]^-{(g^\prime)^*} & GW^{n-1}(\A^n-0)\ar[d]^-{g^*}\ar[r]^-\partial & W^n_{\{0\}}(\A^{n-1})\ar[r]\ar[d] & 0\\
0\ar[r] & GW^{n-1}(\A^n)\ar[r]_-{i^*} & GW^{n-1}(\A^n-0)\ar[r]_-\partial & W^n_{\{0\}}(\A^n)\ar[r] & 0}$$
\begin{lem}
We have $g^*(\theta(x_1,\ldots,x_n))=\rho(\varphi(x_1),\ldots,\varphi(x_n))\cdot \theta(x_1,\ldots,x_n)$ in $\GW {n-1}{\A^n-0}$.
\end{lem}

\begin{proof}
It suffices to compute the element $\partial g^*(\theta(x_1,\ldots,x_n))$ in $W^n_{\{0\}}(\A^n)=W(k)$. This is exactly $\rho(\varphi(x_1),\ldots,\varphi(x_n))$.
\end{proof}

\begin{rem}
Here, we see $\rho(\varphi(x_1),\ldots,\varphi(x_n))$ as an element of $W(k)$ using d\'evissage (see Section \ref{finite_length}).
\end{rem}

\begin{defin}\label{deg}
We call the class of $\rho(\varphi(x_1),\ldots,\varphi(x_n))$ in $W(k)$ the \emph{degree} of $g$ and we denote it by $\mathrm{deg}(g)$.
\end{defin}

\begin{rem}
This terminology is justified as follows: Suppose that $k$ admits a real embedding $k\subset \R$. Pulling back, we get a morphism $g:(\A^n_\R-0)\to (\A^n_\R-0)$ and, taking the rational points, a $C^\infty$ map $g:(\R^n-0)\to (\R^n-0)$. Now one can associate to any $C^\infty$ map $f:(\R^n,0)\to (\R^n,0)$ a topological degree of $f$, denoted by $\mathrm{deg}_{\mathrm{top}}(f)$, as the sum of the signs of the Jacobian of $f$ at all preimages (under $f$) near $0$ of a regular value of $f$ near $0$. In \cite{EL}, the authors show that this degree can be computed as the degree (in the sense of Definition \ref{deg}) of a polynomial map $f^\prime$ approximating $f$. Thus, our degree and the topological degree coincide in that case. 
\end{rem}

\subsection{The counter-example}
Let $n\geq 2$ and $g:(\A^n-0)\to (\A^n-0)$ be the morphism defined by the ring homomorphism $\varphi:k[x_1,\ldots,x_n]\to k[x_1,\ldots,x_n]$ with $\varphi(x_1)=x_1^2-x_2^2$, $\varphi(x_2)=x_1x_2$ and $\varphi(x_i)=x_i$ for $3\leq i\leq n$. Observe that $k[x_1,\ldots,x_n]/(\varphi(x_1),\ldots,\varphi(x_n))$ is a finite length module of length $4$. Following Section \ref{finite_length}, it suffices to compute the class in $W^{fl}(k[x_1,\ldots,x_n]_\mathfrak m)$ of this module endowed with the form
$$\rho:k[x_1,\ldots,x_n]/(\varphi(x_1),\ldots,\varphi(x_n))\to \cha{k[x_1,\ldots,x_n]/(\varphi(x_1),\ldots,\varphi(x_n))}$$
defined by $\rho(1)=Kos(\varphi(x_1),\ldots,\varphi(x_n))$.

\begin{lem}\label{degree}
We have $\mathrm{deg}(g)=\langle 1,1\rangle$.
\end{lem}

\begin{proof}
Let $R=k[x,y]$ and $\mathfrak m=(x,y)$. By d\'evissage, it suffices clearly to prove that $\mathrm{deg}(g)=\langle 1,1\rangle$ if $g:(\A^2-0)\to (\A^2-0)$ is the map induced by $\varphi:R\to R$ with $\varphi(x)=x^2-y^2$ and $\varphi(y)=xy$. We denote by $M$ the ideal generated by the sequence $(x^2-y^2,xy)$ and by $\rho:R/M\to \cha {R/M}$ the symmetric isomorphism $\rho(\varphi(x),\varphi(y))$. Consider the class of $x^2\in R/M$. A direct computation shows that it is non trivial and annihilated by $\mathfrak m$. Moreover, $R/(M,x^2)=R/\mathfrak m^2$ and we then get the following diagram:
$$\xymatrix{0\ar[r] & R/\mathfrak m\ar[r]^-{x^2} & R/M\ar[r]^-\pi\ar[d]^-\rho & R/\mathfrak m^2\ar[r] & 0\\
0\ar[r] & \cha {R/\mathfrak m^2}\ar[r]_-{\cha{\pi}} & \cha {R/M}\ar[r]_-{\cha {x^2}} & \cha {R/\mathfrak m}\ar[r] & 0}$$
We next check that $\cha {x^2}(Kos(x^2-y^2,xy))=Kos(x,y)$. This is clear from the commutative diagram
$$\xymatrix@C=5.3em@R=3em{0\ar[r] & R\ar[r]^-{\begin{pmatrix} -y \\ x\end{pmatrix}}\ar@{=}[d] & R^2\ar[r]^-{\begin{pmatrix} x & y\end{pmatrix}}\ar[d]^-{\begin{pmatrix} x & 0\\ y & x\end{pmatrix}} & R\ar[r]\ar[d]^-{x^2} & R/\mathfrak m\ar[r]\ar[d]^-{x^2} & 0\\
0\ar[r] & R\ar[r]_-{\begin{pmatrix} -xy \\ x^2-y^2\end{pmatrix}} & R^2\ar[r]_-{\begin{pmatrix} x^2-y^2 & xy\end{pmatrix}} & R\ar[r] & R/\mathfrak m\ar[r] & 0.}$$
It follows then that $\cha{x^2}\rho x^2=0$, proving that $x^2:R/\mathfrak m\to R/M$ is a sub-Lagrangian of $\rho:R/M\to \cha {R/M}$. There exists then a homomorphism $\alpha:R/\mathfrak m^2\to \cha {R/\mathfrak m}$ such that the diagram commutes:
$$\xymatrix{  & & & \mathfrak m/\mathfrak m^2\ar[d]^-i & \\
0\ar[r] & R/\mathfrak m\ar[r]^-{x^2}\ar@{-->}[d]^-{\cha\alpha} & R/M\ar[r]^-\pi\ar[d]^-\rho & R/\mathfrak m^2\ar[r]\ar@{-->}[d]^-\alpha & 0\\
0\ar[r] & \cha {R/\mathfrak m^2}\ar[r]_-{\cha{\pi}}\ar[d]^-{\cha i} & \cha {R/M}\ar[r]_-{\cha {x^2}} & \cha {R/\mathfrak m}\ar[r] & 0\\
 & \cha{\mathfrak m/\mathfrak m^2} & & & }$$
By commutativity of the diagram, $\alpha(1)=Kos(x,y)$ and its kernel is $\mathfrak m/\mathfrak m^2$. 
The snake Lemma yields an isomorphism $\psi:\mathfrak m/\mathfrak m^2\to \cha{\mathfrak m/\mathfrak m^2}$ which is easily checked to be symmetric. The sub-Lagrangian reduction (\cite[\S 1.1.5]{Baba}) shows then that $\rho=\psi$ in $W^{fl}(R)$. In order to compute $\psi$, we need a better understanding of the groups and homomorphisms in the diagram. We first compute $\cha {R/\mathfrak m^2}$. 

The projective resolution
$$\xymatrix@C=5.3em{0\ar[r] & R^2\ar[r]^-{\begin{pmatrix} y & 0 \\ -x & -y \\ 0 & x\end{pmatrix}} & R^3\ar[r]^-{\begin{pmatrix} x^2 & xy & y^2\end{pmatrix}} & R\ar[r] & R/\mathfrak m^2\ar[r] & 0     }$$
shows that $\cha {R/\mathfrak m^2}$ is generated by two extensions, the push-out by the homomorphisms $\begin{pmatrix} 1 & 0\end{pmatrix}:R^2\to R$ and $\begin{pmatrix} 0 & 1\end{pmatrix}:R^2\to R$. We will denote by $\epsilon_1$ the first extension and by $\epsilon_2$ the second extension. Now the diagram
$$\xymatrix@C=4.5em@R=3.5em{0\ar[r] & R\ar[r]^-{\begin{pmatrix} -xy \\ x^2-y^2\end{pmatrix}}\ar[d]_-{-x} & R^2\ar[r]^-{\begin{pmatrix} x^2-y^2 & xy\end{pmatrix}}\ar[d]_-{\begin{pmatrix} 1 & 0\\ 0 & 1\\ -1 & 0 \end{pmatrix}} & R\ar[r]\ar@{=}[d] & R/M\ar[r]\ar[d]_-{\pi} & 0\\
0\ar[r] & R\ar[r]_-{\begin{pmatrix} y \\ -x \\0\end{pmatrix}} & R^3/(0,-y,x)\ar[r]_-{\begin{pmatrix} x^2 & xy & y^2\end{pmatrix}} & R\ar[r] & R/\mathfrak m^2\ar[r] & 0    }$$
shows that $\cha\pi(\epsilon_1)=-xKos(x^2-y^2,xy)$. A straightforward computation shows that $\cha\pi(\epsilon_2)=yKos(x^2-y^2,xy)$. 

Now $\mathfrak m/\mathfrak m^2$ is generated by the classes of $x$ and $y$, yielding an isomorphism $(R/\mathfrak m)^2\simeq \mathfrak m/\mathfrak m^2$. We get a projective resolution
$$\xymatrix@C=4.5em{0\ar[r] & R^2\ar[r]^-{\begin{pmatrix} -y & 0\\ x& 0\\ 0 & -y\\ 0 & x\end{pmatrix}} & R^4\ar[r]^-{\begin{pmatrix} x & y & 0 & 0\\ 0 & 0 & x & y\end{pmatrix}} & R^2\ar[r]^-{\begin{pmatrix} x & 0 \\0 & y\end{pmatrix}} & \mathfrak m/\mathfrak m^2\ar[r] & 0}$$
and the commutative diagram
$$\xymatrix@C=4.0em@R=4em{0\ar[r] & R^2\ar[r]^-{\begin{pmatrix} -y & 0\\ x& 0\\ 0 & -y\\ 0 & x\end{pmatrix}}\ar[d]_-{\begin{pmatrix} -1 & 0\end{pmatrix}} & R^4\ar[r]^-{\begin{pmatrix} x & y & 0 & 0\\ 0 & 0 & x & y\end{pmatrix}}\ar[d]_-{\begin{pmatrix} 1 & 0 & 0 & 0\\ 0 & 1 & 1 & 0\\ 0 & 0 & 0 & 1\end{pmatrix}} & R^2\ar[r]^-{\begin{pmatrix} x & 0 \\0 & y\end{pmatrix}}\ar[d]_-{\begin{pmatrix} x & y\end{pmatrix}} & \mathfrak m/\mathfrak m^2\ar[r]\ar[d]_-i & 0 \\
0\ar[r] & R\ar[r]_-{\begin{pmatrix} y \\ -x \\0\end{pmatrix}} & R^3/(0,-y,x)\ar[r]_-{\begin{pmatrix} x^2 & xy & y^2\end{pmatrix}} & R\ar[r] & R/\mathfrak m^2\ar[r] & 0}$$
shows that $\cha i(\epsilon_1)=\begin{pmatrix}-1 & 0\end{pmatrix}$. Similarly, $\cha i(\epsilon_2)=\begin{pmatrix} 0 & 1\end{pmatrix}$. 

Using our computations, we see that $\psi:\mathfrak m/\mathfrak m^2\to \cha{\mathfrak m/\mathfrak m^2}$ is given (in the basis $x,y$) by the matrix $\begin{pmatrix} 1 & 0\\ 0 & 1\end{pmatrix}$. Hence $\mathrm{deg}(g)=\langle 1,1\rangle$.
\end{proof}

\begin{cor}
Let $k$ be a field such that $\sqrt -1\not\in k$ and let $n\geq 3$ be an odd integer. Consider the unimodular row $(x_1,\ldots,x_n)\in Um_n(S_n)$ and the map $g:(\A^n-0)\to (\A^n-0)$ defined by the algebra homomorphism $\varphi:k[x_1,\ldots,x_n]\to k[x_1,\ldots,x_n]$ given by $\varphi(x_1)=x_1^2-x_2^2$, $\varphi(x_2)=x_1x_2$ and $\varphi(x_i)=x_i$ for $3\leq i\leq n$. Then the unimodular row $(x_1^2-x_2^2,x_1x_2,x_3,\ldots,x_n)\in Um_n(S_n)$ is not completable in an invertible matrix.
\end{cor}

\begin{proof}
Consider the following diagram:
$$\xymatrix{S_n\ar[r]^-p & (\A^n-0)\ar[r]^-g & (\A^n-0)}$$ 
Then $(x_1^2-x_2^2,x_1x_2,x_3,\ldots,x_n)$ is the unimodular row associated to the morphism of schemes $gp:S_3\to \A^3-0$. We have 
$$\phi(gp)=(gp)^*(\theta(x_1,\ldots,x_n))=p^*g^*(\theta(x_1,\ldots,x_n)).$$ 
By definition, $g^*(\theta(x_1,\ldots,x_n))=\mathrm{deg}(g)$, which is $\langle 1,1\rangle$ by Lemma \ref{degree}. Then $p^*g^*(\theta(x_1,\ldots,x_n))=\langle 1,1\rangle \cdot \theta(x_1,\ldots,x_n)$ and $\Phi(x_1^2-x_2^2,x_1x_2,x_3,\ldots,x_n)=\langle 1,1\rangle$. Now $\langle 1,1\rangle=0$ in $W(k)$ if and only if $-1$ is a square in $k$. It follows that the unimodular row $(x_1^2-x_2^2,x_1x_2,x_3,\ldots,x_n)$ is not completable.
\end{proof}

We now give the counter-example to M. V. Nori's question. 

\begin{thm}\label{counter}
Let $k$ be a field such that $\sqrt -1\not\in k$. Consider the unimodular row $(x_1,x_2,x_3)\in Um_3(S_3)$ and the map $g:(\A^3-0)\to (\A^3-0)$ defined by the algebra homomorphism $\varphi:k[x_1,x_2,x_3]\to k[x_1,x_2,x_3]$ given by $\varphi(x_1)=x_1^2-x_2^2$, $\varphi(x_2)=x_1x_2$ and $\varphi(x_3)=x_3$. Then $k[x_1,x_2,x_3]/(x_1^2-x_2^2,x_1x_2,x_3)$ is of length $4$, but $(x_1^2-x_2^2,x_1x_2,x_3)\in Um_3(S_3)$ is not completable.
\end{thm}

\begin{proof}
Clear from the above corollary.
\end{proof}

This example shows that Nori's question has to be strengthened. We propose the following:

\begin{quest}[Strong Nori's question]\label{nori+}
Let $R$ be a $k$-algebra, $n\in\N$ be odd and let $f:\spec R\to (\A^n-0)$ be a unimodular row. Let $\varphi:k[x_1,\ldots,x_n]\to k[x_1,\ldots,x_n]$ such that $\varphi(\mathfrak m)\subset \mathfrak m$ and that $(n-1)!$ divides the length of $k[x_1,\ldots x_n]/\varphi(\mathfrak m)$. Let $g:(\A^n-0)\to (\A^n-0)$ the morphism induced by $\varphi$. If the degree $\mathrm{deg}(g)=0$, then the unimodular row $gf:\spec R\to (\A^n-0)$ is completable.
\end{quest}

\begin{thm}
If $\Phi:Um_3(S_3)/SL_3(S_3)\to W(k)$ is injective, then Question \ref{nori+} has an affirmative answer for $n=3$.
\end{thm}

\begin{proof}
Let $f:\spec R\to (\A^3-0)$ be a unimodular row. Then $f$ factors through $S_3$ as indicated in the following diagram:
$$\xymatrix{\spec R\ar[r]^-{f^\prime}\ar[rd]_-f & S_3\ar[d]^-p \\
 & (\A^3-0)}$$
Let $g:(\A^3-0)\to (\A^3-0)$ be a morphism as in the conjecture. Consider the unimodular row $gp:S_3\to (\A^3-0)$. Let $r:\SL 3\to (\A^3-0)$ be the projection to the first row. If $\mathrm{deg}(g)=0$ and $\Phi$ is injective, it follows that $gp:S_3\to (\A^3-0)$ factors through $\SL 3$. We then have a commutative diagram:
$$\xymatrix{ & S_3\ar[d]^-p\ar[rd]^-{\exists q} & \\
\spec R\ar[ru]^-{f^\prime}\ar[r]_-f\ar[rd]_-{gf} & (\A^3-0)\ar[d]^-g & \SL 3\ar[ld]^-r\\
 &(\A^3-0) & }$$
Therefore $gf=rqf^\prime$ factors through $\SL 3$, thus showing that $gf$ is completable.
\end{proof}



\bibliography{biblio_nori.bib}{}
\bibliographystyle{plain}


\end{document}